\documentclass[11pt]{article}
\usepackage{pgfplots}
\pgfplotsset{compat=1.15}
\usepackage{mathrsfs}
\usetikzlibrary{arrows}
\usepackage{amscd,amsmath,amsfonts,amssymb,amsthm,cite,mathrsfs,color} 
\usepackage{dsfont} \usepackage{leftidx}
\usepackage{tikz} \usepackage{enumerate}
\usepackage{appendix}

\usepackage{hyperref}

\usepackage[a4paper,left=25mm, right=25mm, top=25mm, bottom=25mm]{geometry}
\usepackage{graphicx} \usepackage{xcolor}

\numberwithin{equation}{section} 

\usepackage{amsthm}

\theoremstyle{plain}
\newtheorem{theorem}{Theorem}[section]
\newtheorem{lemma}[theorem]{Lemma}

\theoremstyle{definition}

\newtheorem*{acknow}{Acknowledgments}
\begin{document}
	
	\title{\bf{ Edge statistics for singular values of products of rectangular complex Ginibre matrices }}
	
	\author{
		Yandong Gu\footnotemark[1],   
	}
	\renewcommand{\thefootnote}{\fnsymbol{footnote}}
	\footnotetext[1]{School of Mathematical Sciences, University of Science and Technology of China, Hefei 230026, P.R.~China. E-mail: gd27@mail.ustc.edu.cn
		
	}

	\maketitle
\begin{abstract}
We investigate the edge statistics of singular values for products of independent rectangular complex Ginibre matrices. Building on work \cite{LWW23}, which introduced the depth-to-width ratio (DWR) and established the critical kernel,  we show that the DWR acts as a sharp threshold: local  statistics exhibit Gaussian fluctuations in the low-DWR regime, while the Airy kernel emerges at the soft edge in the high-DWR regime.  
\end{abstract}

\section{Introduction and main results}

\subsection{Introduction}
Research on products of random matrices originated in the seminal 1954 work of Bellman \cite{Bel54}. Foundational asymptotic results were subsequently established by Furstenberg and Kesten \cite{FK60} in 1960, who developed laws of large numbers and central limit theorems for these products, thereby extending classical probability theory to non-commutative settings. Meanwhile, the Tracy-Widom distributions, arising from the edge statistics of random matrices\cite{TW94}, have become cornerstones of modern probability theory. 
There has been growing  interest in the spectral properties of random matrix products, motivated by their pivotal role in diverse fields such as wireless communications \cite{TV04}, statistical physics related to chaotic dynamical systems \cite{Cpv93}, and free probability theory \cite{MS17}.

In this paper, we consider products of independent complex Ginibre matrices,
\begin{equation}
	Y_M = X_M \cdots X_1,
\end{equation} where each $X_j$ is a complex Ginibre matrix of size $(N + v_{j}) \times (N + v_{j-1})$ with $v_0 = 0$ and  $v_1, \dots, v_M \geq 0$.
 Akemann, Kieburg and Wei \cite{AKW13} derived the exact joint density of singular values of $Y_M$ and established that these singular values form a determinantal point process. Subsequently, Kuijlaars and Zhang \cite{KZ14} obtained a double integral representation for the correlation kernel. Specifically, for the log-transformed matrices $\log(Y^*_MY_M)$, the correlation kernel is given by 
 \begin{equation}\label{logkn}
 	K_{M,N}(x,y) = \int_{\mathcal{C}}\frac{ds}{2\pi i}\oint_{\Sigma}\frac{dt}{2\pi i}\frac{e^{xt-ys}}{s-t}\frac{\Gamma(t)}{\Gamma(s)}\prod_{j=0}^M\frac{\Gamma(s+N+v_j)}{\Gamma(t+N+v_j)},
 \end{equation}
 where
 $\mathcal{C}$ is a positively oriented Hankel contour in the left half-plane, encircling $(-\infty,-1]$ while starting and ending at $-\infty$, and 
 $\Sigma$ is a closed contour enclosing $[0,N-1]$ chosen disjoint from $\mathcal{C}$.
 With $v_1, \dots, v_M$ being fixed, Liu, Wang and Zhang \cite{LWZ16} established the convergence to sine and Airy kernels for any fixed  $M$ and as $N \to \infty$. Conversely, for fixed $N$ and $M \to \infty$, Akemann, Burda, and Kieburg \cite{ABK14} proved that the $N$ finite-size Lyapunov exponents become asymptotically independent Gaussian random variables.
When both the depth  parameter 
$M$ and  the width parameter 
$N$ tend to infinity,
 Liu, Wang, and Wang \cite{LWW23} rigorously established the local
singular value statistics undergo a phase transition as the depth-to-width ratio (DWR) $M/N$ varies  from $0$ to $\infty$ for square complex  Ginibre matrices (each of size $N \times N$). This  crossover phenomenon was independently identified by Akemann, Burda, and Kieburg \cite{ABK19} in the physics literature.
Specifically,  as $M+N\to \infty$  there exists a universal three-phase transition at the soft edge.
  \begin{itemize}
 \item  {\bf{High DWR (Deep \& Narrow: $M/N \to \infty$)}}:  Gaussian fluctuation appears;

\item {\bf{Moderate DWR (Balanced: $M/N  \to \gamma \in (0,\infty)$) }}:  Critical kernel  arises;

\item  {\bf{Low DWR (Shallow \& Wide: $M/N \to 0$)}}:  Airy kernel appears at the soft edge.
\end{itemize}
A related double-limit phenomenon has been widely studied in diverse random matrix product ensembles \cite{AA22}\cite{AP23}\cite{BE25}\cite{GS22}\cite{GL25}
\cite{Jq17}\cite{LW24}, revealing universal transition patterns in mathematical perspective of spectral statistics. Among these results, the author and Liu established in \cite{GL25} a complete characterization of the phase transition for    local statistics of products of truncated unitary matrices, governed by a modified DWR.

The main goal of this paper   is to study the edge statistics of singular values for products of rectangular complex Ginibre matrices and to   identify a phase transition governed by the  DWR \begin{equation}
	\Delta_{M,N} = \sum\limits_{j=0}^{M} \frac{1}{N + v_j},
\end{equation} introduced in \cite{LWW23}, varies from 0 to $\infty$. 
 In the moderate DWR regime where  $	\lim_{N \to \infty}\Delta_{M,N} = \gamma \in (0,\infty)$, the critical kernel at the soft edge was established in \cite{LWW23}. We will  show that 
$\Delta_{M,N}$
 constitutes a sharp threshold for this transition.

\subsection{Main results}
  For a determinantal point process with correlation kernel $K_{M,N}(x,y)$, recall that the $n$-point correlation functions are given by
 \begin{equation}
 	R_{M,N}^{(n)}(x_1,\dots,x_n) = \det\bigl[K_{M,N}(x_i,x_j)\bigr]_{i,j=1}^n,
 \end{equation} see e.g.\cite{AGZ10}.
 Throughout this work, $\psi(z) = \Gamma'(z)/\Gamma(z)$ denotes the digamma function, while $\Re z$ and $\Im z$ denote the real and imaginary parts of $z$, respectively. We now present our main results.
\begin{theorem} {\bf({Normality})}\label{supthm}
	Suppose that $\lim\limits_{M+N \to \infty} \Delta_{M,N}=\infty$. For any fixed $k \in \mathbb{N}$, let \begin{equation}\rho_{M,N}(k)=\Big(\sum_{j=0}^M\psi'(n+v_j+1-k)\Big)^{\frac{1}{2}},
	\end{equation} and
	\begin{equation}
		x_i=\rho_{M,N}(k)\xi_i+\sum_{j=0}^M\psi(N+v_j+1-k) ,\quad i=1,2,\dots,n.
	\end{equation}
Then  the following limits for correlation functions of eigenvalues of $\log(Y^*_MY_M)$
	\begin{equation}
		\lim_{M+N \to \infty}\big(\rho_{M,N}(k)\big)^n	R_{M,N}^{(n)}(x_1,\dots,x_n) =
		\begin{cases}
			\frac{1}{\sqrt{2\pi}}e^{-\frac{\xi_1^2}{2}} ,&n=1\\
			0,& n>1,
		\end{cases}
	\end{equation}
hold uniformly for $\xi_1,\dots,\xi_n$ in a compact subset of $\mathbb{R}$.
\end{theorem}
In order to state the results for the regime $\lim\limits_{M+N \to \infty}\ \Delta_{M,N}=0$, we need to introduce the scaling parameter 
\begin{equation}
\rho_{M,N}=2^{-\frac{1}{3}}\Big(\frac{1}{z_0^2}-\sum_{j=0}^M\frac{1}{(N+v_j+z_0)^2}\Big)^{-\frac{1}{3}},
\end{equation}
and a spectral parameter
\begin{equation}
	\lambda_M=\frac{1}{z_0}\prod_{j=0}^M(N+v_j+z_0),
\end{equation}
where $z_0$ is the unique positive solution of the rational equation
\begin{equation}\label{defz_0}
	\sum_{j=0}^M\frac{1}{N+v_j+z}- \frac{1}{z}=0.
\end{equation}
We also give a deﬁnition of the celebrated Airy kernel by
\begin{equation}
	\begin{split}
		K_{\mathtt{Ai}}(x,y)
		&=\frac{1}{(2 \pi i)^2}\int_{\gamma_R}du\int_{\gamma_L}d\lambda \frac{e^{u^3/3-xu}}{e^{{\lambda}^3/3-y\lambda}}\frac{1}{u-\lambda},
	\end{split}
\end{equation}
where $\gamma_R$ and $\gamma_L$ are symmetric with respect to the imaginary axis, and $\gamma_R$ is a contour in the right-half plane going from $e^{-\frac{\pi i}{3}}\cdot \infty$ to $e^{\frac{\pi i}{3}}\cdot \infty$, see e.g.\cite{AGZ10}.
\begin{theorem}	{\bf(GUE edge statistics)}\label{subthm}	Suppose that $\lim\limits_{M+N \to \infty} \Delta_{M,N}=0$.\\
Let 
	\begin{equation}\label{subcgd}
		x_i=\frac{\xi_i}{\rho_{M,N}}+\log \lambda_M,\quad i=1,2\dots,n.
	\end{equation}
Then  the following limits for correlation functions of eigenvalues of $\log(Y^*_MY_M)$
\begin{equation}
	\lim_{M+N \to \infty}
(\rho_{M,N})^{-n}	R^{(n)}_{M,N}(x_1,\dots,x_n)=\det\big[K_{\mathtt{Ai}}(\xi_i,\xi_j)\big]_{i,j=1}^n,
\end{equation}
hold uniformly for $\xi_1,\dots,\xi_n$ in a compact subset of $\mathbb{R}$.
\end{theorem}

These results extend random matrix universality to non-square multiplicative chains, with potential implications for stability analysis in deep neural networks and  communication systems. Adapting  a proof strategy analogous to \cite{LWW23}, our central contribution provides a complete characterization of edge statistics through the  parameter $\Delta_{M,N}$, which quantifies the collective rectangularity of the system. Furthermore, in the low-DWR regime, we establish edge limits for matrices with different size via a spectral edge parameterization.

\section{Proofs of Theorems    \ref{supthm} and  \ref{subthm}}\label{sect.2}
The asymptotic analysis of $n$-correlation functions for eigenvalues of $\log(Y_M^*Y_M)$ reduces to the correlation kernel \eqref{logkn}, so we focus here on the kernel. Using steepest descent methods analogous to \cite{LWW23}, we decompose contours and apply Taylor expansions to isolate dominant terms, with the core difficulty being the choice of integration contours and precise estimates of integrals.
 In the proofs, the digama function admits a series representation for $z\neq 0,-1,-2,\cdots$
\begin{equation}
	\psi(z)=-\gamma_0+\sum_{n=0}^{\infty}\big(\frac
	{1}{n+1}-\frac{1}{n+z}\big),\quad \psi'(z)=\sum_{n=0}^{\infty}\frac{1}{(n+z)^2},
\end{equation}
where $\gamma_0$ is the Euler constant; see\cite{OLBC10}.  By Stirling's formula, the following asymptotic expansions hold uniformly as $z \to \infty$ in the sector $|\arg(z)| \leq \pi - \epsilon$ for any $\epsilon > 0$
\begin{equation}\label{eslogg}
	\log\Gamma(z) = \left( z - \frac{1}{2} \right) \log(z) - z + \log\sqrt{2\pi} + \frac{1}{12z} + O\left( \frac{1}{z^2} \right),
\end{equation}
\begin{equation}
	\psi(z) = \log(z) - \frac{1}{2z} + O\left( \frac{1}{z^2} \right).
\end{equation}

\subsection{Proof of Theorem \ref{supthm}}
\begin{proof}[Proof of Theorem \ref{supthm}]
Recall $\rho_{M,N}(k)$ and denote $g(k;\xi)=\sum_{j=0}^M\psi(N+v_j+1-k)+\xi\rho_{M,N}(k)$ in the setting. Direct calculation yields
\begin{align}\label{supker}
	e^{(k-1)(\xi-\eta)\rho_{M,N}(k)}K_{M,N}(g(k;\xi),g(k;\eta))=\int_{\mathcal{C}}\frac{ds}{2 \pi i} \oint_{\Sigma}\frac{dt}{2\pi i} \frac{\Gamma(t)}{\Gamma(s)}\frac
	{1}{s-t}\frac{e^{F(t;k)}}{e^{F(s;k)}}\frac{e^{(t+k-1)\rho_{M,N}(k)\xi}}{e^{(s+k-1)\rho_{M,N}(k)\eta}},
\end{align}
where 
\begin{align}
	F(t;k)=&\Big( \sum_{j=0}^M\psi(N+v_j+1-k)\Big)t-\sum_{j=0}^M \log \frac{\Gamma(t+N+v_j)}{\Gamma(N+v_j)}.
\end{align}

\textbf{Step 1: Contour constructions  and integral decomposition.}

 For the $s$-variable, define a vertical contour $\mathcal{L}_{1-k}$ passing to the right of $1-k$ as
\begin{align}
	&	\mathcal{L}_{1-k}=\Big\{1-k+\frac{2}{\rho_{M,n}(k)}+iy\mid y\in \mathbb{R}\Big\},\\
	&\mathcal{L}_{1-k}^{\text{local}}=\Big\{z \in \mathcal{L}_{1-k} \mid |\Im z|\leq N^{1/4}\rho_{M,N}^{-3/4}(k)\Big\},\quad  \mathcal{L}_{1-k}^{\text{global}}=\mathcal{L}_{1-k} \setminus\mathcal{L}_{1-k}^{\text{local}}.
\end{align}
For the $t$-variable, define the positively oriented contours $\Sigma_{-}(a)$ for $a \in (-N+1, 1)$ and $\Sigma_{-}(b)$ for $b \in (-N+1, \frac{1}{2})$ as
\begin{equation} 
	\begin{split}
		\Sigma_{-}(a)=&\Big\{a-\frac{2-i}{4}t \mid t \in [0,1]\Big\} \cup
		\Big\{a-\frac{2+i}{4}+\frac{2+i}{4}t \mid t \in [0,1]\Big\}
		\cup \\
		& \Big\{-t+\frac{i}{4} \mid t \in [\frac{1}{2}-a,N-\frac{1}{2}]\Big\} \cup \Big\{t-\frac{i}{4} \mid t \in [-N+\frac{1}{2},a-\frac{1}{2}]\Big\}\cup\\ &  \Big\{-N+\frac{1}{2}-it \mid t\in [-\frac{1}{4},\frac{1}{4}]\Big\}.\\
\Sigma_{-}(b)=&\Big\{b+\frac{2-i}{4}t \mid t \in [0,1]\Big\} \cup
		\Big\{b+\frac{2+i}{4}-\frac{2+i}{4}t \mid t \in [0,1]\Big\}
		\cup \\
		& \Big\{t-\frac{i}{4} \mid t \in [b+\frac{1}{2},1]\Big\} \cup \Big\{-t+\frac{i}{4} \mid t \in [-1,-b-\frac{1}{2}]\Big\}\cup\\ &  \Big\{1+it \mid t\in [-\frac{1}{4},\frac{1}{4}]\Big\}.
	\end{split}	
\end{equation}
Let $\Sigma_0(1-k)$ denote the positively oriented circle centered at $1-k$ with radius $\rho_{M,N}^{-1}(k)$. For $k=1,2,3,\dots$, the $t$-contour $\Sigma$ is then defined as
\begin{equation}
	\Sigma = \Sigma_0(1-k) \cup \Sigma_{-}\left( \tfrac{1}{2}-k \right) \cup \Sigma_{+}\left( \tfrac{3}{2}-k \right),
\end{equation}
where $\Sigma_{+}\left( \tfrac{3}{2}-k \right)$ is empty when $k=1$.(see figure \ref{fig}.)
\begin{figure}
	\centering
	\includegraphics{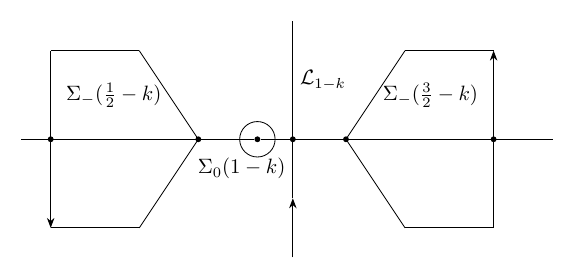}
	\caption{Schematic contours in proof of Theorem \ref{supthm}}
	\label{fig}
\end{figure}
\begin{equation}
	e^{(k-1)(\xi-\eta)\rho_{M,N}(k)}K_{M,N}(g(k;\xi),g(k;\eta))=I_1+I_2,
\end{equation}
where \begin{align}	&I_2=\int_{\mathcal{L}_{1-k}}\frac{ds}{2\pi i}\oint_{\Sigma_{-}(\frac{1}{2}-k) \cup \Sigma_{+}(\frac{3}{2}-k)}\frac{dt}{2\pi i}\frac{\Gamma(t)}{\Gamma(s)}\frac
	{1}{s-t}\frac{e^{F(t;k)}}{e^{F(s;k)}}\frac{e^{(t+k-1)\rho_{M,n}(k) \xi}}{e^{(s+k-1)\rho_{M,n}(k) \eta}},
\end{align}
and for $* = \text{local}$, $\text{global}$, or $\text{blank}$ (denoting the full contour), define \begin{align}
	&	I_1^{*}=\int_{\mathcal{L}^{*}_{1-k}}\frac{ds}{2\pi i}\oint_{\Sigma_0(1-k)}\frac{dt}{2\pi i}\frac{\Gamma(t)}{\Gamma(s)}\frac
	{1}{s-t}\frac{e^{F(t;k)}}{e^{F(s;k)}}\frac{e^{(t+k-1)\rho_{M,N}(k) \xi}}{e^{(s+k-1)\rho_{M,N}(k) \eta}}.\end{align}

\textbf{Step 2:  Asymptotics of $I_1^{\text{local}}$.}

Making the change of variables
\begin{equation}\label{cov}
	s = 1 - k + \frac{\sigma}{\rho_{M,N}(k)}, \quad
	t = 1 - k + \frac{\tau}{\rho_{M,N}(k)},
\end{equation}
we obtain the simplification
\begin{equation}
	\frac{e^{(t+k-1)\rho_{M,N}(k) \xi}}{e^{(s+k-1)\rho_{M,N}(k) \eta}} = e^{\xi \tau - \eta \sigma}.
\end{equation}
Furthermore, for $\sigma, \tau$ in any compact subset of $\mathbb{C} \setminus {0}$ and as $\Delta_{M,N} \to \infty$, we have
\begin{equation}
	\frac{\Gamma(t)}{\Gamma(s)} = \frac{\sigma}{\tau} \left( 1 + O \left( (|\tau| + |\sigma|) \rho_{M,N}^{-1}(k) \right) \right), \quad
	\frac{1}{s-t} = \frac{\rho_{M,N}(k)}{\sigma - \tau}.
\end{equation}
To estimate $I_1^{\text{local}}$, we use the following derivatives of $F(t;k)$.
\begin{align}
	F'(t;k) &= \sum_{j=0}^M \Big( \psi(N+v_j+1-k) - \psi(t + N + v_j) \Big), \\
	F''(t;k) &= -\sum_{j=0}^M \psi'(t + N + v_j),
\end{align}
which satisfy
\begin{equation}
	F'(1-k;k) = 0, \quad F''(1-k;k) = -\rho_{M,N}^2(k).
	\end{equation}
For $t \in \Sigma_0(1-k)$, applying Taylor expansion at $1-k$, we have
\begin{equation}\label{Ftestimate}
	\begin{split}
		F(t;k)=&F(1-k;k)+\frac{1}{2}F''(1-k;k)\big(\frac{\tau}{\rho_{M,N}(k)}\big)^2+O\Big(\frac{\tau^3}{N\rho_{M,N}(k)}
		\Big)\\
		&=F(1-k;k)-\frac{\tau^2}{2}+O\big(N^{-1/4}\rho_{M,N}^{-1/4}(k)\big),
	\end{split}
\end{equation} uniformly for $\tau = \tau' N^{1/4} \rho_{M,N}^{1/4}(k)$ with $\tau'$ in any compact subset of $\mathbb{C}$. Similarly, 
 this estimate \eqref{Ftestimate} also holds for $F(s;k)$ when $s \in \mathcal{L}^{\text{local}}_{1-k}$.
Combining \eqref{Ftestimate} for $F(t;k)$ and $F(s;k)$ with other elementary estimates,  we obtain
\begin{align}
	I_1^{\text{local}}=\frac{1}{\rho_{M,N}(k)}\int_{2-iN^{1/4}\rho_{M,N}^{1/4}(k)}^{2+iN^{1/4}\rho_{M,N}^{1/4}(k)}\frac{d\sigma}{2\pi i}\oint_{|\tau|=1}\frac{d\tau}{2\pi i} \frac{e^{-\frac{\tau^2}{2}+\xi\tau}}{e^{-\frac{\sigma^2}{2}+\eta\sigma}}\frac{\tau}{\sigma}\frac{1}{\sigma-\tau}\Big(1+O(N^{-1/4}\rho_{M,N}^{-1/4}(k))\Big).
\end{align}

\textbf{Step 3: Negligible Contributions.} 

We now establish the asymptotic negligibility of \(I_1^{\text{global}}\) and \(I_2\). For \(s \in \mathcal{L}_{1-k}^{\text{global}}\) parametrized as \(s = 1 - k + (2 + i y)\rho_{M,N}^{-1}(k)\) with \(|y| > N^{1/4} \rho_{M,N}^{1/4}(k)\), consider the derivative  
\begin{equation}
	\frac{d}{dy} \Re F(c + iy; k) = -\Im F'(s; k) \big|_{s=c+iy}
	= \sum_{j=0}^M \Im \psi(N + v_j + c + iy).
\end{equation} 
This derivative implies the existence of \(\epsilon > 0\) such that for all \(|y| \geq N^{1/4} \rho_{M,N}^{1/4}(k)\),
\begin{equation}\label{esref(s;k)}
	\Re F(s; k) \geq \Re F(s_{\pm}; k) + \epsilon N^{1/4} \rho_{M,N}^{1/4}(k) \left( |y| - N^{1/4} \rho_{M,N}^{1/4}(k) \right),
\end{equation}
where \(s_{\pm} = 1 - k + \left(2 \pm i N^{1/4} \rho_{M,N}^{1/4}(k)\right) \rho^{-1}_{M,N}(k)\).  
Combining \eqref{Ftestimate} with the values at \(s_{\pm}\), we obtain  
\begin{equation}\label{esf(s+-)}
	\Re \left( F(s_{\pm}; k) - F(1 - k; k) \right) = \frac{1}{2} \sqrt{N}  \rho_{M,N}^{1/2}(k) \big(1 + o(1)\big).
\end{equation}
Applying \eqref{esref(s;k)} and \eqref{esf(s+-)} yields the asymptotic bound  
\begin{equation}
	\begin{split}
		I_1^{\text{global}}=&\frac{1}{\rho_{M,N}(k)}e^{-\frac{1}{2}N^{1/4}\rho_{M,N}^{1/4}(k)\big(1+o(1)\big)}\int_{\Re \sigma=2,|\Im \sigma|>N^{1/4}\rho_{M,N}^{1/4}(k)}\frac{d\sigma}{2\pi i}\oint_{|\tau|=1} \frac{d\tau}{2\pi i}\frac{\Gamma(t)}{\Gamma(s)}\frac{e^{\xi\tau}}{e^{\eta\sigma}}\frac{1}{\tau-\sigma}\\
		&O\big(e^{-\epsilon n^{1/2}\rho_{M,N}^{1/2}(k)(|\Im \sigma|-N^{1/4}\rho_{M,N}^{1/4}(k))}\big).\end{split}
\end{equation}
To estimate $I_2$, we require the following lemma.
\begin{lemma}\label{supergollemma}
	There exists $\epsilon>0$, such that the inequality
	\begin{equation}
		\Re\big(F(t;k)-F(1-k;k)\big)\leq -\epsilon\rho_{M,N}^2(k)|t-1+k|
	\end{equation}
	holds for all $t$ on $\Sigma_{-}(\frac{1}{2}-k) \cup \Sigma_{+}(\frac{3}{2}-k)$.
\end{lemma}
Combining Lemma \ref{supergollemma} with the estimates \eqref{Ftestimate} for $s\in \mathcal{L}^{\text{local}}_{1-k}$, and \eqref{esref(s;k)}, \eqref{esf(s+-)} for $s\in \mathcal{L}^{\text{global}}_{1-k}$, and noting that $|s-t|>1/4$ for sufficiently large $\Delta_{M,N}$, we obtain for some $\epsilon>0$,
\begin{equation}\label{esI2}
	\begin{split}
		I_2=	&\int_{\mathcal{L}_{1-k}}\frac{ds}{2\pi i}\oint_{\Sigma_{-}(\frac{1}{2}-k) \cup \Sigma_{+}(\frac{3}{2}-k)}\frac{dt}{2\pi i}\frac{\Gamma(t)}{\Gamma(s)}	\frac{e^{(t+k-1)\rho_{M,N}(k) \xi}}{e^{(s+k-1)\rho_{M,N}(k) \eta}}
		\\ 
		&\times \begin{cases}
			\big(1+O(N^{-1/4}\rho_{M,N}^{-1/4}(k))\big)e^{-\frac{F''(1-k;k)(s-1+k)^2}{2}},& |\Im s|\leq N^{1/4}\rho_{M,N}^{-3/4}(k),\\
			e^{-\frac{1}{2}\sqrt{N}\rho_{M,N}^{1/2}(k)-\epsilon N^{-1/4}\rho_{M,N}^{3/4}(k){\big(|\Im s|-N^{1/4}\rho_{M,N}^{-3/4}(k)\big)}},& \text{otherwise}.\\
		\end{cases}
	\end{split}
\end{equation}
Applying the steepest descent method to  $I_1^{\text{global}}$ and $I_2$, we have the asymptotic bounds
\begin{equation}\label{negterm}
	|I_1^{\text{global}}|=o\big(e^{-\frac{1}{4}\sqrt{N}\rho_{M,N}^{1/2}(k)}\big),\quad |I_2|=o\big(e^{-\frac{\epsilon}{4}\rho_{M,N}^2(k)}\big).\end{equation} 
As $N \to \infty$ and $ \Delta_{M,N} \to \infty$,  we have
\begin{equation}\label{esI1}
	I_1=\Big(1+O\big(\rho_{M,N}^{-1}(k)\big)\Big)\frac{1}{\rho_{M,N}(k)}\int_{2-i\infty}^{2+i\infty}\frac{d\sigma}{2\pi i}\oint_{|\tau|=1}\frac{d\tau}{2\pi i}\frac{e^{-\frac{\tau^2}{2}+\xi\tau}}{e^{-\frac{\sigma^2}{2}+\eta\sigma}}\frac{\tau}{\sigma}\frac{1}{\sigma-\tau}.
\end{equation}
From \eqref{esI2},\eqref{negterm} and \eqref{esI1}, we conclude that \(I_1\) is dominated by \(I_1^{\text{local}}\) while \(I_2\) is asymptotically negligible.

Combining with 
$\lim\limits_{M+ N\to \infty}\rho_{M,N}(k) I_1=\frac{1}{\sqrt{2\pi}}e^{-\eta^2/2}$, we thus complete the proof of  Theorem \ref{supthm}.
\end{proof}

\begin{proof}[Proof of Lemma \ref{supergollemma}]
 For $t \in B(1-k, 1+k)$, we have
\begin{equation}
	\rho_{M,N}^{-2}(k)\big(F(t;k) - F(1-k;k)\big) = -\frac{1}{2} \big(t - (1-k)^2\big) + O(N^{-1}),
\end{equation}
and $\frac{d}{dx} \Re F\big(-k \pm i/4;k\big) > 0$. Consequently, the function $\rho_{M,N}^{-2}(k) \Re \big(F(t;k) - F(1-k;k)\big)$ decreases as $t$ moves away from $1-k$ along the segment $B(1-k,1+k) \cap \left[ \Sigma_{-}\big(\frac{1}{2}-k\big) \cup \Sigma_{+}\big(\frac{3}{2}-k\big) \right]$. On the horizontal components of $\Sigma_{-}\big(\frac{1}{2}-k\big)$, 
\begin{align}
	\frac{d^2}{dx^2} \Re F\big(x \pm \tfrac{i}{4};k\big) 
	= \Re \left. \frac{d^2F(t;k)}{dt^2} \right|_{t=x\pm\frac{i}{4}} 
	= - \sum_{j=0}^M \Re \psi'\big(x+N+v_j \pm \tfrac{i}{4}\big)< 0.
\end{align}
it follows that $\Re \left[ \rho_{M,N}^{-2}(k) F(t;k) \right]$ decreases as $t$ moves away from $1-k$, since $\frac{d}{dx} \Re F\big(x \pm \frac{i}{4};k\big) > 0$ at the endpoints of these horizontal contours.
Finally, for $t = -N + \frac{1}{2} + iy$ with $y \in [-\frac{1}{4}, \frac{1}{4}]$, a direct calculation yields
\begin{equation}
	\Re \big(F(t;k) - F(1-k;k)\big) = -\rho_{M,N}^{2}(k) (N+1-k)^2 \big(1 + o(1)\big).
\end{equation}
Combining these estimates gives the desired result.
\end{proof}

\subsection{Proof  of Theorem \ref{subthm}}
\begin{proof}[Proof of Theorem \ref{subthm}]
After changing the variables $s,t \to \rho_{M,N}^{3/2}s,\rho_{M,N}^{3/2}t$, we obtain
\begin{equation}
K_{M,N}(x,y)=\rho_{M,N}^{3/2}\int_{\mathcal{C}}\frac{ds}{2 \pi i} \oint_{\Sigma}\frac{dt}{2\pi i}\frac{e^{\rho_{M,N}^{3/2}(xt- ys)}}{s-t}
	\prod_{j=0}^{M}\frac{\Gamma(\rho_{M,N}^{3/2}s+N+v_j)}{\Gamma(\rho_{M,n}^{3/2}t+N+v_j)}.
\end{equation}
Making the substitution $ g(\xi)=\xi\rho^{-1}_{M,N}+\log \lambda_M$, we find that
\begin{equation}
K_{M,N}(g(\xi),g(\eta))=\rho_{M,N}^{3/2}\int_{\mathcal{C}}\frac{ds}{2 \pi i} \oint_{\Sigma}\frac{dt}{2\pi i}\frac{e^{\rho_{M,N}^{1/2}(\xi t-\eta s)}}{s-t}e^{F_{M,N}(\rho_{M,N}^{3/2}s)-F_{M,N}(\rho_{M,N}^{3/2}t)},
\end{equation}
where 
\begin{equation}
	F_{M,N}(s)=\sum_{j=0}^M\log \frac{\Gamma(s+N+v_j)}{\Gamma(N+v_j)}-\log\Gamma(s)-s\log(\lambda_M).
\end{equation}
Applying Stirling's formula for $\log \Gamma(z)$ as given in \eqref{eslogg}, we obtain the asymptotic expansion
\begin{equation}
F_{M,N}(\rho_{M,N}^{3/2}s)=\rho_{M,N}^{3/2}f_{M,N}(s)+O(\Delta_{M,N}),
\end{equation}
where 
\begin{align}
	f_{M,N}(s)=\sum_{j=0}^M\frac{N+v_j}{\rho_{M,N}^{3/2}}\Big(1+\frac{\rho_{M,N}^{3/2}s}{N+v_j}\Big)\log \Big(1+\frac{\rho_{M,N}^{3/2}s}{N+v_j}\Big)-s(\log s
	-v_{M,N}),
\end{align}
with $v_{M,N}=\sum_{j=0}^M\log(N+v_j)-\log\rho_{M,N}^{3/2}-\log \lambda_M$. 

\textbf{Step 1: Contour constructions  and integral decomposition.}

We deform integration contours to isolate the critical point, noting $q_0 := z_0 \rho_{M,N}^{-3/2}$.
Deform the contour for $s$ as $\mathcal{C}=\mathcal{C}_{\text{local}} \cup \mathcal{C}_{\text{global}}$, where  $\mathcal{C}_{\text{local}}=\mathcal{C}_{\text{local},+} \cup \mathcal{C}_{\text{local},-}$,  with
\begin{align}
&\mathcal{C}_{\text{local},+}=\Big\{q_0+\big(1/\rho_{M,N}^{3/2}\big)^{1/3}+re^{\frac{4\pi i}{3}}\mid -\big(1/\rho_{M,N}^{3/2}\big)^{3/10}\leq r \leq 0\Big\},\\
&\mathcal{C}_{\text{local},-}=\Big\{q_0+\big(1/\rho_{M,N}^{3/2}\big)^{1/3}+re^{\frac{5\pi i}{3}}\mid 0\leq r \leq \big(1/\rho_{M,N}^{3/2}\big)^{3/10}\Big\},
\end{align}
and the global contour
\begin{align}
\mathcal{C}_{\text{global}}=&\Big\{q_0+\big(1/\rho_{M,N}^{3/2}\big)^{1/3}-\big(1/\rho_{M,N}^{3/2}\big)^{1/3}e^{\frac{4 \pi i}{3}}+iy \mid y\geq 0\Big\}\notag \\
&\cup \Big\{q_0+\big(1/\rho_{M,N}^{3/2}\big)^{1/3}+\big(1/\rho_{M,N}^{3/2}\big)^{1/3}e^{\frac{5 \pi i}{3}}+iy \mid y\leq 0\Big\}.
\end{align}
Deform the $t$-contour as $\Sigma = \Sigma_{\text{local}} \cup \Sigma_{\text{global}}$, where $\Sigma_{\text{local}} = \Sigma_{\text{local},+} \cup \Sigma_{\text{local},-}$ with
\begin{align}
	&\Sigma_{\text{local},+}=\Big\{q_0-\big(1/\rho_{M,N}^{3/2}\big)^{1/3}+re^{\frac{2\pi i}{3}}\mid 0\leq r \leq \big(1/\rho_{M,N}^{3/2}\big)^{3/10}\Big\},\\
	&\Sigma_{\text{local},-}=\Big\{q_0-\big(1/\rho_{M,N}^{3/2}\big)^{1/3}+re^{\frac{\pi i}{3}}\mid  -\big(1/\rho_{M,N}^{3/2}\big)^{3/10}\leq r \leq 0\Big\}.
\end{align}
The global contour $\Sigma_{\text{global}}$ is constructed to ensure $\Re f_{M,N}(t)$ attains its minimum at the left endpoints of $\Sigma_{\text{local},\pm}$. It is defined as
\begin{equation}
\Sigma_{\text{global}}:=\Sigma_{+}^1 \cup \Sigma_{+}^2 \cup \Sigma_{+}^3 \cup \Sigma^4 \cup  \Sigma_{-}^1 \cup \Sigma_{-}^2 \cup \Sigma_{-}^3,
\end{equation}
where 
\begin{align}
	&\Sigma_{\pm}^1= \Big\{\text{vertical line connecting the left end of $\Sigma_{\text{local},\pm}$ to $(x_1,\pm y_1)$}\Big\}, \notag\\
	&\Sigma_{\pm}^2=\Big\{t=x \pm iy \mid \frac{y-y_1}{y_2-y_1}=\frac{x-x_1}{x_2-x_1} \Big\},\notag\\
	&\Sigma_{\pm}^3=\Big\{t=x \pm iy_2 \mid x\in[-N/\rho_{M,N}^{3/2}+(2\rho_{M,N}^{3/2})^{-1},x_2]\Big\},\notag\\
	&\Sigma^4=\Big\{t=-N/\rho_{M,N}^{3/2}+(2\rho_{M,N}^{3/2})^{-1}+iy \mid y \in [-y_2,y_2]\Big\}.
\end{align}
The contours are oriented such that $\Sigma$ is positively oriented, while $\mathcal{C}$ is from $-i\infty$ to $i\infty$;
see figure \ref{subcontour}.
	\begin{figure}
	\centering
	\includegraphics{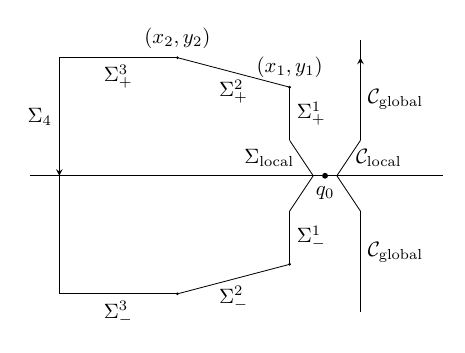}
	\caption{Schematic contours in proof of Theorem \ref{subthm}}
	\label{subcontour}
\end{figure}

In the following, we provide the selection strategy for the contour $\Sigma_{\text{global}}$ and show that $\Re f_{M,N}(t) $ attains its minimum at left end of $\Sigma_{\text{local},\pm}$ on $\Sigma_{\text{global}}$.
For $t \in \Sigma_{\pm}^1$, we set 
$x_1=q_0-\big(\rho_{M,N}^{-3/2}\big)^{1/3}-1/2\big(\rho_{M,N}^{-3/2}\big)^{3/10}$,
 and choose $y_1$ satisfies
$\frac{d\Re f_{M,N}(x_1+iy_1)}{dy}=0$.
Next, we choose  $C$ as a large enough positive constant, independent of $M,N$ and let $x_2=-C$, such that  $\frac{d\Re f_{M,N}(x_2+iy)}{dx}<0$ for $t\in \Sigma^3_{+}$ and choose $y_2$ such that $\frac{d \Re f_{M,N}(x_2+iy_2)}{dy}=0$.
At last, we illustrate that $\Re f_{M,N}(t)$ increases as $t$ along the $(x_1,y_1)\to (x_2,y_2)$.
Note that $z=re^{i\theta}$, $r>0$ and  the angle $\theta \in (\pi/2,\pi)$ is fixed.
By the first derivative 
\begin{equation}
	\begin{split}\frac{d\Re f_{M,N}(x_1+iy_1+z)}{dr}
		=\cos\theta \Re f'_{M,N}(x_1+iy_1+z)-\sin \theta \Im f'_{M,N}(x_1+iy_1+z),
	\end{split}
\end{equation}
and second derivative 
\begin{equation}\label{urr}
	\frac{d^2\Re f_{M,N}(x_1+iy_1+z)}{dr^2}=-\frac{1}{r}\frac{d\Re f_{M,N}(x_1+iy_1+z)}{dr}-\frac{1}{r^2}	\frac{d^2\Re f_{M,N}(x_1+iy_1+z)}{d\theta^2},
\end{equation}
we see that the signs of first-order and second-order derivatives are opposite. Combining \eqref{urr} and  $\frac{d\Re f_{M,N}(x_1+iy_1)}{dr}>0$ and $\frac{d\Re f_{M,N}(x_2+iy_2)}{dr}>0$, we can  conclude that $\frac{d\Re f_{M,N}(t)}{dr}>0$ for $t\in \Sigma^2_{+}$.
 Similarly, we can prove that 
$\Re f_{M,N}(t)$ increases as $t$ moves upward on $\Sigma^1_{+}$, or t moves downward on $\Sigma^1_{-}$ and  $\Re f_{M,N}$ increases as $t$ moves leftward, and attains its minimum at $(x_2,\pm y_2)$ for $t\in \Sigma^3_{\pm}$.  It is easy to see that $\Re f_{M,N}(t)$ attains its maximun at $t=-N/\rho_{M,N}^{3/2}+(2\rho_{M,N}^{3/2})^{-1}$ for $t \in \Sigma^4$.
Therefore, the integral decomposes into a local part  $\mathcal{C}_{\text{local}}\times \Sigma_{\text{local}} $  and a global part  $\mathcal{C}\times \Sigma \setminus (\mathcal{C}_{\text{local}}\times \Sigma_{\text{local}})$.

\textbf{Step 2: Asymptotics of the local part.}

We compute derivatives to analyze the local behavior,
\begin{align}
	&	f_{M,N}'(s)=\sum_{j=0}^M \log\Big(1+\frac{\rho_{M,N}^{3/2}s}{N+v_j}\Big)-\log s +v_{M,N},\\
	&f_{M,N}''(s)=\rho_{M,N}^{3/2}\sum_{j=0}^M\frac{1}{N+v_j+\rho_{M,N}^{3/2}s},\\
	&f_{M,N}'''(s)=-\rho_{M,N}^3\sum_{j=0}^M\frac{1}{(N+v_j+\rho_{M,N}^{3/2}s)^2}.
\end{align}
Recalling that $z_0$ defined in \eqref{defz_0} and  $q_0= z_0 \rho_{M,N}^{-3/2}$, we have
\begin{equation}
	f_{M,N}'(q_0)=0,\quad 	f_{M,N}''(q_0)=0,
	\quad	f_{M,N}'''(q_0)=2+o(1).
\end{equation}
After makig the change of variables \begin{equation}
t=q_0+\rho_{M,N}^{-1/2}\tau, \quad s=q_0+\rho_{M,N}^{-1/2}\sigma,
 \end{equation}
 and taking Taylor expansion at $q_0$
 in $B(q_0,\rho_{M,N}^{-\frac{3}{2}\times \frac{3}{10}})$, we obtain
\begin{equation}
	f_{M,N}(t)=f_{M,N}(q_0)+\frac{1}{6}f_{M,N}'''(q_0)(t-q_0)^3+O(\rho_{M,N}^{-\frac{3}{2}\times \frac{6}{5}}),
\end{equation}
The local part estimate is given by
\begin{align}\label{localestimate}
	&\big(\rho_{M,N}^{1/2}\big)e^{-\rho_{M,N}^{1/2}(\xi-\eta)q_0}\int_{\mathcal{C}_{\text{local}}}\frac{ds}{2\pi i}\int_{\Sigma_{\text{local}}}\frac{dt}{2 \pi i}e^{(\xi t -\eta s)\rho_{M,N}^{1/2}}\frac{e^{ \rho_{M,N}^{3/2}\frac{f_{M,N}'''(q)}{6}\big((s-q)^3-(t-q)^3\big)+O(\rho_{M,N}^{-\frac{3}{10}})}}{s-t}\notag \\
	&=\int_{\mathcal{C}^{\rho_{M,N}^{1/20}}_{<}}\frac{d\sigma}{2\pi i} \int_{\Sigma_{>}^{\rho_{M,N}^{1/20}}}\frac{d \tau}{2\pi i }\frac{1}{\sigma-\tau}\frac{e^{\frac{\sigma^3}{3}-\eta \sigma}}{e^{\frac{\tau^3}{3}-\xi \tau}}\big(1+O(\rho_{M,N}^{-\frac{3}{10}})\big) \\ \notag
	&=K_{\mathtt{Ai}}(\xi,\eta)+O(\rho_{M,N}^{-\frac{3}{10}}),
\end{align}
where the contours for real $R > 0$ are defined as upward oriented 
\begin{align}
	&\mathcal{C}_{<}^R=\{1+re^{\frac{4\pi i}{3}}\mid -R\leq r\leq 0\} \cup \{1+re^{\frac{5\pi i}{3}}\mid 0\leq r\leq R\},\\
	&\Sigma_{>}^R=\{-1+re^{\frac{2\pi i}{3}}\mid 0\leq r\leq R\} \cup \{-1+re^{\frac{\pi i}{3}}\mid -R\leq r\leq 0\}.
\end{align}

\textbf{Step 3:  Estimates of the global part.}

For $s$, we consider the infinite contour $\mathcal{C}_{\text{global}}$ split into two parts: $\mathcal{C}^1_{\text{global}} = \{ s \in \mathcal{C}{\text{global}} \mid |\Im s| \leq K \}$ and $\mathcal{C}_{\text{global}}^2 = \mathcal{C}_{\text{global}} \setminus \mathcal{C}_{\text{global}}^1$, where $K$ is a large positive constant.

\begin{lemma}\label{globalref}
	Let $x_0>q_0$, and  $C=\{x_0+iy\mid y\in \mathbb{R}\}\}$, then $\Re f_{M,N}(s)$ attains its global maximum at $x_0$ on $C$. Moreover, 
	\begin{equation}
		\frac{d \Re f_{M,N}(x_0+iy)}{dy}\begin{cases}
			<0, &y>0,\\
			>0,  & y<0.\\
		\end{cases}
	\end{equation}
\end{lemma}
By Lemma \ref{globalref}, $\Re f_{M,N}(s)$ decreases as $s$ moves upward on $\mathcal{C}_{\text{global}}^1$  $ \cap $    $\mathbb{C}_{+}$ or $s$ moves downward on $\mathcal{C}_{\text{global}}^1$ $\cap$ $\mathbb{C}_{-}$.
Direct computation shows that $\Im f'_{M,N}(t)>1$ if $s\in \mathcal{C}_{\text{global}}^2 \cap \mathbb{C}_{+}$ and $\Im f'_{M,N}(t)<-1$ if $s\in \mathcal{C}_{\text{global}}^2 \cap \mathbb{C}_{-}$. Hence $f_{M,N}(s)$ decreases at least linearly fast as $s$ moves to $\pm i \infty$ along $\mathcal{C}_{\text{global}}^2$.
For $t \in \Sigma_{\text{global}}$,  $\Re f_{M,N}(t;\lambda_M)$ attains a minimum at the leftend of $\Sigma_{\text{local},\pm}$.
Combining these global estimates for 
 $s\in \mathcal{C}_{\text{global}}$ and $t \in \Sigma_{\text{global}}$, there exists $\epsilon>0$ such that 
\begin{equation}\label{subglobales}
\Re \Big(f_{M,N}(s)-\frac{\xi(t-q_0)}{\rho_{M,N}}\Big)+\epsilon \rho_{M,N}^{-27/20}\\
	\leq \Re \Big(f_{M,N}(t)-\frac{\eta (t-q_0)}{\rho_{M,N}}\Big)-\epsilon \rho_{M,N}^{-27/20}.
\end{equation}
Thus, combining  \eqref{subglobales} with the estimates of 
$f_{M,N} (t; \lambda_M)$ on
$\Sigma_{\text{local}}$	 and $\mathcal{C}_{\text{local}}$
from \eqref{localestimate} yields
\begin{equation}\label{globalestimate}
	\begin{split}
	&	\rho_{M,N}^{-1}e^{-(\xi-\eta)\frac{z_0}{\rho_{M,N}}}\int\int_{\mathcal{C}\times\Sigma\setminus \mathcal{C}_{\text{local}}\times \Sigma_{\text{local}}}\frac{ds}{2\pi i}\frac{dt}{2\pi i}\frac{1}{s-t}e^{\rho_{M,N}^{1/2}(\xi t-\eta s)}e^{\rho_{M,N}^{3/2}(f_{M,N}(s)-f_{M,N}(t))}\\
	&=\rho_{M,N}^{-1}\int\int_{\mathcal{C}\times\Sigma\setminus \mathcal{C}_{\text{local}}\times \Sigma_{\text{local}}}\frac{ds}{2\pi i}\frac{dt}{2\pi i}\frac{1}{s-t}
	\frac{e^{\rho_{M,N}^{3/2}\Big(f_{M,N}(s)-\frac{\eta(s-q_0)}{\rho_{M,N}}\Big)}}{e^{\rho_{M,N}^{3/2}\Big(f_{M,N}(t)-\frac{\xi(t-q_0)}{\rho_{M,N}}\Big)}}
=O\big(e^{-\epsilon(\rho_{M,N}^{-3/20})}\big).
	\end{split}
\end{equation}

Combining the ``local" estimate \eqref{localestimate} and ``global" estimate \eqref{globalestimate}, we  complete the proof of  Theorem \ref{subthm}.
\end{proof}
	
	\begin{proof}[Proof of Lemma \ref{globalref}]
		We analyze the derivative of the real part of $f_{M,N}$ along the imaginary axis
		\begin{align}
			\frac{d}{dy}\Re f_{M,N}(x_0 + iy) 
			&= \Im \bigg( \log z 
			- \sum_{j=0}^M \log\Big(1+\frac{\rho_{M,N}^{3/2}z}{N+v_j}\Big) \bigg) \bigg|_{z=x_0+iy} \notag\\
			&= -\rho_{M,N}^{3/2}y \sum_{j=0}^M  \frac{1}{N+v_j + \rho_{M,N}^{3/2}x_0} 
	 + O(\Delta_{M,N}).
		\end{align}
		The key observation comes when $\rho_{M,N}^{3/2}x_0 > z_0$, where $z_0$ is the solution to equation \eqref{defz_0}, implying
$	\sum_{j=0}^M  (N+v_j + \rho_{M,N}^{3/2}x_0)^{-1}  > 0$. This inequality yields $\frac{d}{dy}\Re f_{M,N}(x_0+iy) < 0$ for $y > 0$
and $\frac{d}{dy}\Re f_{M,N}(x_0+iy) > 0$ for $y < 0$.
Consequently, $\Re f_{M,N}(s)$ attains its global maximum at $x_0$ on the contour $C$.
	\end{proof}

\begin{acknow}
 The author would like to express  his deep gratitude to Dang-Zheng Liu for his valuable advice. This work was supported by the National Natural Science Foundation of China \#12371157.
\end{acknow}

	\end{document}